\documentclass[a4paper,11pt]{amsart}
\usepackage{amsfonts,amssymb}
\usepackage{graphicx,tikz}
\usepackage{tikz-cd} 
\usepackage{float}
\usepackage{systeme}
\usepackage{relsize}
\usepackage{amsmath, amsthm, amsfonts}
\usepackage{geometry}
\usepackage{verbatim}
\usepackage{hyperref}
\usepackage{enumitem}  
\usepackage[capitalise]{cleveref}
\usepackage{comment}
\usepackage{enumitem}
\usepackage{amsrefs}
\usepackage[official]{eurosym}
\usepackage{graphicx}
\usepackage{nomencl}
\usepackage{amsfonts}
\usepackage{hyperref}
\usepackage{amssymb}
\usepackage{fancyhdr}
\usepackage{amscd}
\usepackage[english]{babel}

\theoremstyle{plain}
\theoremstyle{definition}

\newtheorem*{theorem*}{Theorem}
\newtheorem{thm}{Theorem}[section]
\newtheorem{pr}[thm]{Proposition}

\newtheorem{lem}[thm]{Lemma}

\theoremstyle{definition}
\newtheorem*{dfn*}{Definition}

\newtheorem*{rem*}{Remark}

\makeatletter
\def\cleardoublepage{\clearpage\if@twoside \ifodd\c@page\else
	\hbox{}
	\thispagestyle{empty}
	\newpage
	\if@twocolumn\hbox{}\newpage\fi\fi\fi}
\makeatother

\DeclareMathOperator{\Sym}{Sym}
\DeclareMathOperator{\Id}{Id}
\DeclareMathOperator{\ad}{ad}

\keywords{Engel elements, Lie algebras}
\subjclass[2010]{20F45, 20F40}

\begin{document}
	
\title[Left 3-Engel elements in locally finite $p$-groups]{Locally finite $p$-groups with a left 3-Engel element whose normal closure is not nilpotent}
\author[A. Hadjievangelou]{Anastasia Hadjievangelou}
\address{Department of Mathematical Sciences, University of Bath, Claverton Down, Bath BA2 7AY, United Kingdom}
\email{ah926@bath.ac.uk}
\author[M. Noce]{Marialaura Noce}
\address{Department of Mathematical Sciences, University of Bath, Claverton Down, Bath BA2 7AY, United Kingdom}
\email{mn670@bath.ac.uk}
\author[G. Traustason]{Gunnar Traustason}
\address{Department of Mathematical Sciences, University of Bath, Claverton Down, Bath BA2 7AY, United Kingdom}
\email{gt223@bath.ac.uk}

\thanks{}

\begin{abstract}
For any odd prime $p$, we give an example of a locally finite $p$-group $G$ containing a left 3-Engel element $x$ where $\langle x \rangle^G$ is not nilpotent.
\end{abstract}

\maketitle

\section{Introduction}

\mbox{}\\
Let $G$ be a group. An element $a\in G$ is a left Engel element in $G$, if for
each $x\in G$ there exists a non-negative integer $n(x)$ such that
      $$[x,_{n(x)}a]=[[[x,\underbrace {a],a],\ldots ,a]}_{n(x)}=1.$$
If $n(x)$ is bounded above by $n$ then we say that $a$ is a left $n$-Engel element
in $G$. It is straightforward to see that any element of the Hirsch-Plotkin
radical $HP(G)$ of $G$ is a left Engel element and the converse is known
to be true for some classes of groups, including solvable groups and 
finite groups (more generally groups satisfying the maximal condition on
subgroups) [4,1]. The converse is however not true in general and this is the case
even for bounded left Engel elements. In fact whereas one sees readily that 
a left $2$-Engel element is always in the Hirsch-Plotkin radical this
is still an open question for left  $3$-Engel elements. Recently there has been a breakthrough and in [7] it is shown that any left $3$-Engel element of odd order is contained in $HP(G)$. From [12] one also knows that in order to generalise this to left $3$-Engel elements of any finite order it suffices to deal with elements of
order $2$. \\ \\
It was observed by William
Burnside [3] that every element in a group of exponent $3$  is a left $2$-Engel
element and so the fact that every left $2$-Engel element lies in the Hirsch-Plotkin radical can be seen as the underlying reason why groups of exponent
$3$ are locally finite. For groups of $2$-power exponent there is a close link
with left Engel elements. If  $G$ is a group of exponent 
$2^{n}$ then it is not difficult to see that any element $a$ in $G$ of order $2$ is a left $(n+1)$-Engel element of $G$ (see the introduction of [13] for details). For sufficiently large $n$ we know that the variety of groups of exponent $2^{n}$ is not locally finite [6,8]. As a result one can see (for example in [13]) that it follows that for sufficiently large $n$ we do not have in general that a left $n$-Engel element is contained in the Hirsch-Plotkin radical. Using the fact that groups of exponent $4$ are locally finite [11], one can also see that if all left $4$-Engel elements of a group $G$ of exponent $8$ are in
$HP(G)$ then $G$ is locally finite. \\ \\
Swapping the role of $a$ and $x$ in the definition of a left Engel element we get the notion of a right Engel element. Thus an element $a\in G$ is a right Engel element, if for each $x\in G$ there exists a non-negative 
integer $n(x)$ such that 
    $$[a,_{n(x)} x]=1.$$
If $n(x)$ is bounded above by $n$, we say that $a$ is a right $n$-Engel element. By a classical result
of Heineken [5] one knows that if $a$ is a right $n$-Engel element in $G$ then $a^{-1}$ is a left $(n+1)$-Engel
element. \\ \\
In [9] M. Newell proved that if $a$ is a right $3$-Engel element in $G$ then $a\in HP(G)$ and in
fact he proved the stronger result that $\langle a\rangle^{G}$ is nilpotent of class at most $3$. The natural question arises whether the analogous result holds for left $3$-Engel elements. In [10] it has been shown that this is not the case by giving an example of a locally finite $2$-group with a left $3$-Engel element $a$ such that $\langle a\rangle^{G}$ is not nilpotent.
In this paper we extend this result to include any odd prime. The odd case turns out to be more involved and the construction in many ways quite different from the $p=2$ case. \\ \\
The structure of the paper is as follows. In Section 2 we describe the pair $(G,x)$ that will provide our example and we show directly that $x$ is a left $3$-Engel element in $G$. In order to show that $\langle x\rangle^{G}$ is not nilpotent we need more work. In Sections 3 and 4 we construct a pair $(L,z)$ where $L$ is a Lie algebra over ${\mathbb F}_{p}$, the field of $p$ elements, and where $\mbox{Id}(z)$ is not nilpotent. The pair $(L,z)$ can be seen as the Lie algebra version of our group construction and in Section 5
we build a group $H$ within $\mbox{End}(L)$ containing $1+\mbox{ad}(z)$ where $(1+\mbox{ad}(z))^{H}$ is not nilpotent and where $(H,1+\mbox{ad}(z))$ is a homomorphic image of $(G,x)$. 

\section{The group $G$}
\mbox{}\\
Let $x, a_{2}, a_{2},\ldots $ be group variables. We recall that a simple commutator in $x, a_{1}, a_{2}, \ldots $ is a group word defined recursively as follows: $x, a_{1}, a_{2}, \ldots $ are simple commutators and if $u,v$ are simple commutators then $[u,v]$ is a simple commutator.  We say that a simple commutator $s$ has multi-weight $(m,e_{1},e_{2},\ldots )$ in $x, a_{1}, a_{2}, \ldots $, if $x$ occurs $m$ times and $a_{i}$ occurs $e_{i}$ times in $s$. The weight of $s$ is then $m+e_{1}+e_{2}+\cdots $. The following definition will also play a crucial role.
\begin{dfn*}
Let $s$ be a simple commutator of multi-weight $(m,e_{1},e_{2},\ldots )$ in  $x,a_{1},a_{2},\ldots $. The {\it type} of $s$ is $t(s)=e_{1}+e_{2}+\ldots -2m$.
\end{dfn*}

\begin{rem*}
If $u,v$ are simple commutators in $x,a_{1},a_{2},\ldots $ then $t([u,v])=t(u)+t(v)$. In particular $t([u,a_{j}])=t(u)+1$
and $t([u,x])=t(u)-2$. 
\end{rem*} 
\mbox{}\\
For a fixed odd prime $p$, let $G= \langle x, a_1, a_{2}, \dots \rangle$ be the largest group satisfying the following relations:
\begin{enumerate}
    \item $\langle a_i \rangle^G$ is abelian for all $i\geq 1$;
    \item  $\langle x \rangle^G$ is metabelian;
    \item $x^{p}=a_{1}^{p}=a_{2}^{p}=\cdots =1$;
    \item if $s \neq x$ is a simple commutator in $x,a_{1},a_{2},\ldots $, where $t(s) \geq 2$ or $t(s) \leq -2$, then $s=1$.
\end{enumerate}
{\bf Remark}. Notice that $G$ is a locally finite $p$-group. Also, from (4) it follows immediately that if $s$ is a simple commutator in $x,a_{1},a_{2},\ldots$, then $[s,x]=1$ whenever $t(s)\not =1$. This is something we will make use  of later.
\\ \\
%
Let $s$ be a left-normed commuator of weight $3m+1$ in  $x,a_{1},a_{2},\ldots $, starting in $x$.  The reader can convince himself/herself that from (4) it follows that such a commutator will be trivial if it is not of the form $[x,a_{j_{1}},a_{j_{2}},a_{j_{3}},x,a_{j_{4}},a_{j_{5}},\ldots $ $,x,a_{j_{2m}},a_{j_{2m+1}}]$. We will establish later that there are such non-trivial commutators for any $m$ that implies that $\langle x\rangle ^{G}$ is not nilpotent. However the structure of $G$ is not transparent enough at this stage for us to see this directly. Instead  will will  use Lie ring methods to lift up this property from an analogous Lie ring setting that we will deal with in the next two sections. On the other hand we can establish directly that $x$ is a left $3$-Engel element and this we will do now for the rest of this section. We start with a useful lemma.
\begin{lem} Let $g\in G$. We can write $g=\alpha\beta\gamma\delta$ where $\alpha, \beta, \gamma, \delta$ are products of simple commutators in $x,a_{1},a_{2},\ldots $ of types $1,2,-1,0$.
\end{lem}
\begin{proof} We prove first by induction on $k\geq 1$ that we can write $g=\alpha_{k}\beta_{k}\gamma_{k}\delta_{k}\epsilon(k+1)$,
where $\alpha_{k},\beta_{k},\gamma_{k},\delta_{k}$ are products of simple commutators in $x,a_{1},a_{2},\ldots $ of types $1,2,-1,0$ and $\epsilon(k+1)$  is a product of commutators of weight $k+1$ or higher. For the induction basis observe that we can
write $g=x_{1}x_{2}\cdots x_{m}$ where $x_{j}\in \{x,a_{1},a_{2},\ldots \}$ and modulo commutators of weight $2$ or higher we can rearrange the product so that $g=\alpha_{1}\beta_{1}\epsilon(2)$ and where $\alpha_{1}$ is a product of elements from $\{a_{1},a_{2},\ldots \}$, $\beta_{1}$ is a power of $x$ and $\epsilon(2)$ is a product of simple commutators of weight $2$ or higher (here $\gamma_{1}=\delta_{1}=1$). For the induction step suppose that $k\geq 2$ and that the inductive claim holds for smaller values of $k$. Thus we know that we can write $g=\alpha_{k-1}\beta_{k-1}\gamma_{k-1}\delta_{k-1}\epsilon(k)$ where $\alpha_{k-1},\beta_{k-1},\gamma_{k-1},\delta_{k-1}$ are products of simple commutators in $x,a_{1},a_{2},\ldots $ of types $1,2,-1,0$  and where $\epsilon(k)$ is a product of simple commutators of weight $k$ or higher. Working modulo commutators of weight $k+1$ and higher, we can now collect the commutators of weight $k$ so that they appear with the commutator of smaller weight that are of the same type. This
gives us the expression $g=\alpha_{k}\beta_{k}\gamma_{k}\delta_{k}\epsilon(k+1)$ that we wanted.
This finishes the proof of the inductive hypothesis. Notice now that we are working with simple commutators in $x_{1},x_{2},\ldots ,x_{m}$ and as $H=\langle x_{1},\ldots ,x_{m}\rangle$ is nilpotent we have that for a large enough $k$ we have $\epsilon(k+1)=1$. 
\end{proof}
Let $g\in G$. We want to show that $[g,x,x,x]=1$. Let $g=\alpha\beta\gamma\delta$ be an expression for $g$ as in Lemma 2.1. 
As we remarked above, all simple commutators of type different from $1$ commute with $x$. This means that $\beta,\gamma,\delta$ commute with $x$ and without loss of generality we can assume that 
                       $$g=b_{1}\cdots b_{m}$$
where $b_{1},\ldots ,b_{m}$ are simple commutators in $x,a_{1},a_{2},\ldots $ of type $1$. \\ \\
Let $H=\langle x,b_{1},\ldots ,b_{m}\rangle$.
Observe that relations (1) and (2)  imply that $\langle b_{i}\rangle^{H}$ is abelian and $\langle x\rangle^{H}$ is metabelian. Notice
also that  relations (1) and (3) imply that $b_{1}^{p}=\cdots =b_{m}^{p}=1$. Finally, suppose that we have a simple commutator $c$
in $x,b_{1},\ldots ,b_{m}$ of multi-weight $(l,f_{1},f_{2},\ldots ,f_{m})$ in $x,b_{1},\ldots ,b_{m}$ then the type of $c$ as a simple commutator in $x,a_{1},a_{2},\ldots $ is $-2l+f_{1}t(b_{1})+\cdots + f_{m}t(b_{m})=-2l+f_{1}+\cdots +f_{m}$ that is equal to the type of $c$ as a simple commutator in $x,b_{1},\ldots ,b_{m}$. Hence relations (4) hold for $H$ with $x,a_{1},a_{2},\ldots $
replaced by $x,b_{1},\ldots ,b_{m}$. It follows from this that there is a homomorphism $\alpha:G\rightarrow H$ that maps $x,a_{1},\ldots ,a_{m}$ to $x,b_{1},\ldots ,b_{m}$ and $a_{k}$ to $1$ for $k\geq m+1$. In particular $[b_{1}\cdots b_{m},x,x,x]=1$ would
follow  from $[a_{1}\cdots a_{m},x,x,x]=1$. Thus the problem of showing that $x$ is a left $3$-Engel element reduces to showing that $[a_{1}\cdots a_{r},x,x,x]=1$ for all integers $r\geq 1$. Before turning to this we need to introduce some more standard notation.\\

\noindent {\bf Definition.} Let $y,c_{1},c_{2},...,c_{m}$ be group elements. For any non-trivial subset $I=\{i_{1},\ldots ,i_{k}\}$ of $\{1,2,\ldots ,m\}$, where $i_{1}<\ldots <i_{k}$ we denote by {\it $[y,c_{I}]$} the commutator $[y,c_{i_{1}},\ldots ,c_{i_{k}}]$.

\begin{pr} 
We have $[a_{1}\cdots a_{r},x,x,x]=1$ for all integers $r\geq 1$.
\end{pr}
\begin{proof} We prove this by induction on $r$. For the induction basis simply observe that $t([a_{1},x,x,x])=-5$ and thus $[a_{1},x,x,x]=1$ by relations (4). Now suppose $r\geq 2$ and that the result holds for all smaller values for $r$. First we see from standard
commutator properties that 
 \begin{equation*}
[a_{1}\cdots a_{r},x] = \left(\prod_{I\subseteq \{2,\ldots ,r\}}[a_{1},x,a_{I}]\right)\cdot \left(\prod_{I\subseteq \{3,\ldots ,r\}}[a_{2},x,a_{I}]\right) \dots [a_{r},x].
\end{equation*}
Notice that relations (1) imply that the order of factors within each sub-product above does not matter. The RHS has $s=2^{r}-1$ factors. Suppose these are  $b_{1}, \ldots , b_{s}$ and thus $[a_{1}\cdots a_{r},x]=b_{1}\cdots b_{s}$.  We want to show that 
$[b_{1}\cdots b_{s},x,x]=1$ or equivalently that $[x,b_{1}\cdots b_{s}]$ commutes with $x$. Expanding again we see that 
     $$[x,b_{1}\cdots b_{s}]=\prod_{\emptyset \not =I\subseteq \{1,\ldots ,s\}}[x,b_{I}].$$
Notice that each factor is in $(\langle x\rangle^{G})'$ and as $\langle x\rangle^{G}$ is metabelian the order of the factors does not matter. Notice also that every non-trivial factor is of the form $[x,[x,a_{I_{1}}]^{-1},\ldots ,[x,a_{I_{l}}]^{-1}]$ where, because of relations (1), the sets $I_{1},\ldots , I_{l}$ are pairwise disjoint. We also have that $\mbox{min}(I_{1})<\mbox{min}(I_{2})<\ldots <\mbox{min}(I_{l})$. Let $A$ be the product of such factors where not all $a_{1},a_{2},\ldots ,a_{r}$ are included and let $B$
be the product of those factors where each $a_{1},a_{2},\ldots ,a_{r}$ occurs once. Then 
            $$[x,b_{1}\cdots b_{s}]=AB.$$
It is not difficult to see by standard arguments that the induction hypothesis implies that $A$ commutes with $x$. We are thus left
with showing that $B$ commutes with $x$. Now let 
$$
[x,[x,a_{I_{1}}],\ldots ,[x,a_{I_{m}}]]^{(-1)^{m}}
$$
be one of the factors 
of $B$. This will commute with $x$ if the type is not $1$. As $I_{1}\cup\cdots \cup I_{m}=\{1,2,\ldots ,r\}$ we have that
the  $t([x,[x,a_{I_{1}}],\ldots ,[x,a_{I_{m}}]])=r-2(m+1)$ that is equal to  $1$ if and only if $r=2m+3$. We are thus  left with 
showing that $e(m)$ commutes with $x$ where 
                       $$e(m)=\prod_{(I_{1},I_{2},\ldots ,I_{m})}[x,[x,a_{I_{1}}],\ldots ,[x,a_{I_{m}}]].$$
Here the product is taken over all partitions of $\{1,2,\ldots ,2m+3\}$ into a disjoint union of $m$ non-empty subsets where 
               $$1=\mbox{min}(I_{1})<\mbox{min}(I_{2})<\ldots <\mbox{min}(I_{m}).$$
As $\langle x\rangle^{G}$ is metabelian the value of a factor does not change if we permute the terms $[x,a_{I_{2}}],\ldots ,[x,a_{I_{m}}]$. Notice also that for $[x,[x,a_{I_{1}}]]$ to be non-trivial we need $|I_{1}|=3$. If some $I_{j}$ has size less than $2$
where $2\leq j\leq m$ then (because we can permute the terms $[x,a_{I_{2}}],\ldots ,[x,a_{I_{m}}]$ with out changing the value) we
see that the factor is trivial as $[x,[x,a_{I_{1}}],[x,a_{I_{j}}]]$ has type less than $-1$. For a factor of $e(m)$ to be non-trivial we thus need all of $I_{2},I_{3},\ldots ,I_{m}$ to have size $2$ or $3$. Suppose $s$ are of size $2$ and $t$ of size $3$. Then
the we get the equations
  \begin{eqnarray*}
                     s+t & = & m-1 \\
                  2s+3t & = & 2m.
\end{eqnarray*}
That has the solution $t=2$ and $s=m-3$ (in particular $m\geq 3$). This shows that 
       $$e(m)=\prod_{(I_{1},I_{2},\ldots ,I_{m})}[x,[x,a_{I_{1}}],\ldots ,[x,a_{I_{m}}]]$$
where the product is taken over all partitions of $\{1,2,\ldots ,2m+3\}$ where $|I_{1}|=|I_{2}|=|I_{3}|=3$ and
$|I_{4}|=\ldots =|I_{m}|=2$ and $1\in I_{1}$. \\ \\
Consider any factor of this product. As $[x,a_{I_{4}}],\ldots ,[x,a_{I_{m}}]$ commute with $x$, we have
     $$[e(m),x]=\prod_{(I_{1},I_{2},\ldots ,I_{m})}[x,[x,a_{I_{1}}],[x,a_{I_{2}}],[x,a_{I_{3}}],x,[x,a_{I_{4}}],\ldots ,[x,a_{I_{m}}]].$$
It follows from this that in order to show that $[e(m),x]=1$ for all $m\geq 1$, it suffices to show  that $[e(3),x]=1$.  \\ \\
To see this take any $2\leq i<j<k<r\leq 9$. We have 
        $$1=[x,a_{1},[x,a_{i},a_{j},a_{k},a_{r}]]=\prod_{(I_{1},I_{2})}[x,a_{I_{1}},x,a_{I_{2}}].$$ 
where the product is taken over all partitions of $\{1,i,j,k,r\}$ where $|I_{1}|=3$, $|I_{2}|=2$ and $1\in I_{1}$. From this it follows
that 
       $$e(3)=\prod_{(I_{1},I_{2},I_{3})}[x,[x,a_{I_{1}}],[x,a_{I_{2}}],[x,a_{I_{3}}]]$$
is trivial and thus  in particular $[e(3),x]=1$. This finishes the proof. 
\end{proof}
\mbox{}\\
As a consequence we have the following theorem. 
\begin{thm}
The element $x$ is a left $3$-Engel element in $G$.
\end{thm}

\section{The Lie algebra $L$}
\mbox{}\\
In this section we will construct a pair $(L,z)$ that will provide us with the Lie algebra analogue of $(G,x)$. Most of this section will be about describing the Lie algebra $L$ and getting some close transparent information about $\mbox{Id}(z)$, the ideal in $L$ generated by $z$. This will then be used in Section 4 to prove that $\mbox{Id}(z)$ is not nilpotent. Throughout the paper $p$ is a fixed odd prime and ${\mathbb F}={\mathbb F}_{p}$ is the field of $p$ elements. \\ \\ 
We will start with a result that is probably well known although we do not have a reference. For that reason and for the convenience of the reader we will include a proof. We first need some notation. Let $z,c_{1},c_{2},\ldots $ be some Lie algebra variables. For any non-trivial multi-set $\{c_{i_{1}},c_{i_{2}},\ldots ,c_{i_{k}}\}$ with $i_{1}\leq i_{2}\cdots \leq i_{k}$ we
denote by $[z,c_{I}]$ the left normed Lie commutator $[z,c_{i_{1}},\ldots ,c_{i_{k}}]$. 
\begin{pr} Let $E$ be the largest Lie algebra over ${\mathbb F}$ generated $z,c_{1},c_{2},\ldots $ where $c_{1},c_{2},\ldots $ commute. Then $\mbox{Id}(z)$ is freely generated by $\langle [z,c_{I}],\,I \subset_{fm} {\mathbb N^{m}}\rangle$, where $I$ runs through all finite multi-subsets of ${\mathbb N^{m}}$. Here ${\mathbb N^{m}}$ is the multi-set variant of the set of natural numbers namely $\{1,1,\ldots ,2,2,\ldots ,\ldots\}$. 
\end{pr}
\begin{proof}
Let $A$ be the enveloping algebra of $E$, that is the largest associative algebra generated by $z,c_{1},c_{2},\ldots $ where $c_{1},c_{2},\ldots $ commute. Thus $E$ is the Lie sub-algebra generated by $z,c_{1},c_{2},\ldots$ where the Lie product of $u,v\in A$ is $[u,v]=uv-vu$. In order to prove the proposition it suffices to show that the associative sub-algebra $C$ generated by the set $\{[z,c_{I}],\,I \subset_{fm} {\mathbb N^{m}}\}$ is freely generated as an associative algebra by these elements. 
To see this, pick a variable $y_{I}$ for each finite multi sub-set of ${\mathbb N}^{m}$ and consider the free associative algebra $B$, freely generated by these. Order the generators of $A$ with $c_{1}<c_{2}<\cdots <z$ and use these to define a lexicographical order on words in these. We also order the generators $y_{I}$ by 
                             $$y_{I}<y_{J}\mbox{ if and only if }zc_{i_{1}}\cdots c_{i_{r}}<zc_{j_{1}}\cdots c_{j_{s}}$$
where $I=\{i_{1},\ldots ,i_{r}\}$, $J=\{j_{1},\ldots ,j_{s}\}$, $i_{1}\leq \cdots \leq i_{r}$ and $j_{1}\leq \cdots \leq j_{s}$. 
We are here using the lexicographical order above for the displayed inequalities. Notice that if we expand $[z,c_{I}]$ then the
leading term with respect to that lexicographical order is $zc_{i_{1}}\cdots c_{i_{r}}$. We will denote the latter by $zc_{I}$. We then use this order do define a lexiographical order on words in the $y_{I}$'s. 
Now consider the algebra homomorphism from $B$ to $A$ that maps $y_{I}$ to $[z,c_{I}]$. In order to show that $C$ is freely generated by the $[z,c_{I}]$'s it suffices to show that this homomorphism is injective. Thus let $y$ be a non-zero element in $B$ say
                $$y=\alpha y_{I_{1}}\cdots y_{I_{r}}+\cdots +\beta y_{J_{1}}\cdots y_{J_{s}}$$
where $\alpha y_{I_{1}}\cdots y_{I_{r}}$ is the leading term of $y$ with respect to the lexicographical ordering on $B$ described above. From the definition of this lexicographical order, we see that 
               $$\bar{y}=\alpha zc_{I_{1}}\cdots zc_{I_{r}}+\cdots +\beta zc_{J_{1}}\cdots zc_{J_{s}}$$
has leading term $\alpha zc_{I_{1}}\cdots zc_{I_{r}}$ with respect to the lexicographical order on $A$. But notice that this is also the leading term of the image of $y$ under the homomorphism, namely
           $$\tilde{y}=\alpha [z,c_{I_{1}}]\cdots [z,c_{I_{r}}]+\cdots +\beta [z,c_{J_{1}}]\cdots [z,c_{J_{s}}].$$
This shows that the image of $y$ is non-zero and thus the homomorphism is injective. This finishes the proof. 
\end{proof}
\mbox{}\\
Thus $\mbox{Id}_{E}(z)$ is a free Lie algebra, freely generated by the $[z,c_{I}]$'s. For the following we will use a slightly different total order on this set of free generators. If $I$ and $J$ are both non-empty, then $[z,c_{I}]<[z,c_{J}]$ as above. However if one of the multi-sets is empty we define the order differently so that $z$ is the largest element. That is $[z,c_{I}]<z$ whenever $I$ is non-empty. Now consider the largest metabelian quotient,  $M(z)$, of $\mbox{Id}(z)$. This is the free metabelian Lie algebra
over ${\mathbb F}$ with free generators the $[z,c_{I}]$'s.  Now it is well known (see for example [2]), that $M(z)$ has a basis consisting of the left-normed Lie commutators 
      $$
[[z, c_{I_1}],[z, c_{I_2}], \dots, [z, c_{I_m}]]
$$
where $m \geq 1$ and  $[z, c_{I_1}] > [z, c_{I_2}] \leq \dots \leq [z, c_{I_m}]$. \\ \\
We are now getting ready to describe our Lie algebra $L$. First we consider the  largest Lie algebra $F=\langle z, c_1, c_2, \dots \rangle$ over $\mathbb{F}_p$, such that 
\begin{enumerate}
    \item $\Id(c_i)$ is abelian for $i=1,2, \dots$;
    \item $\Id(z)$ is metabelian.
\end{enumerate}
From the analysis above, the following is a basis for $\Id_F(z)$ 
\begin{equation}\tag{*}\label{*}
[[z, c_{I_1}],[z, c_{I_2}], \dots, [z, c_{I_m}]]
\end{equation}
where $m \geq 1$, $I_1, \dots, I_m$ are pairwise disjoint and $[z, c_{I_1}] > [z, c_{I_2}]\leq \dots \leq [z, c_{I_m}]$. 
We define a type of a Lie commutator in a similar way as we did for group commutators in Section 1. This term will play an important role in this section.

\begin{dfn*}
 Let $s$ be a simple commutator of multi-weight $(m,e_1, \dots, e_r)$ in $z, c_1, \dots, c_r$. The {\it type} of $s$ is $t(s)=e_1+\dots +e_r-2m$.
\end{dfn*}

\begin{rem*} If $c$ and $d$ are simple commutators then $t([c,d])=t(c)+t(d)$. In particular
$t([c,x])=t(c)-2$ and $t([c,a_{1}])=t(c)+1$. 
\end{rem*}
\mbox{}\\
We will now construct a Lie algebra $L$ that will be a quotient of $F$ by a certain multi-homogeneous ideal $J$. Note that since $F$ is multi-graded, $L$ will also be multi-graded.
Our multi-homogeneous ideal will be the smallest ideal of $F$ containing the following elements from our basis for $\Id_{F}(z)$ described above. Let $c=[[z,c_{I_1}],[z,c_{I_2}],\dots,[z,c_{I_m}]]$ be one of these basis elements. \\
\begin{enumerate}[resume]
    \item $c$ is in $J$ if $c \neq z$ is a commutator of type either less than -1 or more than 1;
    \item $c$ is in $J$ if one of $I_1,I_3, \dots, I_m$ has size greater than 2.
\end{enumerate}
In this and following section we will determine the structure of $J$ and prove that $\Id_L(z)$ is not nilpotent.

\subsection{Some detailed analysis about which basis elements are in $J$} \label{3.1}
One can see directly that a number of the commutators in ($\ast$) will be in $J$ as a consequence of conditions (3) and (4). In this section, we will now look more closely at these commutators and see that a number of elements are in $J$ as a less direct consequence.
Consider one of the basis elements, $[[z, c_{I_1}],[z, c_{I_2}], \dots, [z, c_{I_m}]]$, from ($\ast$) above. Note that, in particular, by condition (3) we have that $[z, c_{I}]$ is in $J$  when $|I| \geq  4$. \label{1}
In the following we will denote with $\alpha$ the number of $I_i$ of size 1 and $\beta$ those of size 2, for $i\geq 3$.
\begin{rem*} Notice that, since $\Id(z)$ is metabelian,  it does not matter how we order $I_3, \dots, I_m$ as they can be permuted without changing the value of the Lie commutator. In fact, the only constraint is that the smallest $i$ in $ I_1 \cup \dots \cup I_m$ should be in $I_2$, since $[z, c_{I_1}] > [z, c_{I_2}] \leq \dots \leq [z, c_{I_m}]$.
\end{rem*}
\noindent{\bf Commutators of type -1 where $|I_1|+|I_2|=3$.}\label{3.1.1}
In this case we have $t([z, c_{I_1}], [z, c_{I_2}])=-1$ and $|I_3|+\dots+|I_m|=2m-4$. Notice that none of $I_3, \dots, I_m$ is empty or of size 1 since otherwise the commutator would be in $J$ by condition (3). Then all commutators are in $J$ except possibly those where

$$|I_1|+|I_2|=3, |I_3|=\dots=|I_m|=2, |I_1|\leq 2.$$
\mbox{} \\
\noindent{\bf Commutators of type -1 where $|I_1|+|I_2|=4$.}
Here $t([z, c_{I_1}], [z, c_{I_2}])=0$ and $|I_3|+\dots+|I_m|=2m-5$. As before, none of $I_3, \dots, I_m$ is empty since otherwise the commutator would be in $J$ by condition (3). Furthermore none of the sets is of size greater than 2 by condition (4). We thus obtain the following system of equations
$$
\systeme*{\alpha + \beta = m-2, \alpha + 2\beta = 2m-5}
$$
whose solution is $\alpha = 1$ and $\beta=m-3$.
Therefore all commutators are in $J$ except possibly those where
$$|I_1|+|I_2|=4,  |I_3|=1, |I_4|=\dots=|I_m|=2, |I_{1}|\leq 2.$$

\noindent {\bf Commutators of type -1 where $|I_1|+|I_2|=5$.}
We have $t([z, c_{I_1}], [z, c_{I_2}])=1$ and $|I_3|+\dots+|I_m|=2m-6$. First, we consider the case when $I_3$ is empty. In this case $t([z, c_{I_1}], [z, c_{I_2}],[z, c_{I_3}])=-1$ and none of  $I_4, \dots, I_m$ are of size 1. Thus they are all of size 2, and all commutators are in $J$ except possibly the cases
$$|I_1|=2, |I_2|=3, |I_3|=0, |I_4|+\dots+|I_m|=2.$$  
Now it remains to deal with the case for which none of $I_3, \dots, I_m$ is empty. Using $\alpha$ and $\beta$ as before, we have
$$
\systeme*{\alpha + \beta = m-2, \alpha + 2\beta = 2m-6}
$$
whose solution is $\alpha = 2$ and $\beta=m-4$. Thus all commutators are in $J$ except possibly the cases where
$$|I_1|=2, |I_2|=3, |I_3|=|I_4|=1, |I_5|=\dots=|I_m|=2.$$

\noindent {\bf Commutators of type 0 where $|I_1|+|I_2|=3$.}
In this case none of $I_3, \dots, I_m$ are empty or of size 1, as  otherwise the commutator is in $J$ by condition (3). Also, none are of size 3 by condition (4). Therefore, they must all be of size 2, which implies that the type is -1, that is a contradiction since we supposed that the type is 0. \\

\noindent{\bf Commutators of type 0 where $|I_1|+|I_2|=4$.}
We have that none of $I_3, \dots, I_m$ is empty. Using $\alpha$ and $\beta$ as before, we have
$$
\systeme*{\alpha + \beta = m-2, \alpha + 2\beta = 2m-4}
$$
whose solution is $\alpha = 0$ and $\beta=m-2$. Thus all commutators are in $J$ except possibly those where

$$|I_1|+|I_2|=4, |I_3|=\dots=|I_m|=2, |I_1|\leq 2.$$

\noindent{\bf Commutators of type 0 where $|I_1|+|I_2|=5$.}
First, notice that none of $I_3, \dots, I_m$ is empty. Indeed, if $I_3$ is empty then $t([z, c_{I_1}], [z, c_{I_2}],[z, c_{I_3}])=-1$ and the commutator has type at most $-1$ if it is non-trivial. We then get a contradiction since we are considering the case when the type is 0. Thus none of $I_3, \dots, I_m$ is empty. Therefore we have
$$
\systeme*{\alpha + \beta = m-2, \alpha + 2\beta = 2m-5}
$$
whose solution is  $\alpha = 1$ and $\beta=m-3$. Therefore all commutators are in $J$ except possibly those where

$$|I_1|=2, |I_2|=3, |I_3|=1, |I_4|=\dots=|I_m|=2.$$

\noindent {\bf Commutators of type 1.}\label{3.1.7}
Since $|I_3|+ \dots |I_m| \leq 2m-4$,  we need $|I_1|+|I_2|=5$. Then all commutators are in $J$ except possibly those where
$$|I_1|=2, |I_2|=3, |I_3|=\dots=|I_m|=2.$$
\mbox{}\\
In view of the analysis above we come up with the following decomposition of $\Id_F(z)$ as a direct sum of sub spaces.  
$$
\Id_F(z)=\langle X \setminus Z \rangle \oplus \langle Z \rangle
$$
where $Z$ is the list of basis elements with the following partitions $(I_{1},\ldots ,I_{m})$
\begin{align*}
& [z, c_I], \text{ for } |I| \leq 3\\
    &|I_1|+|I_2|=3, |I_1| \leq 2, |I_3|= \dotsm =|I_m|=2 & (\zeta_1)  \\
    &|I_1|+|I_2|=4, |I_1| \leq 2, |I_3|=1, |I_4|= \dotsm =|I_m|=2 & (\zeta_2)\\
    &|I_1|=2, |I_2|=3, |I_3|=0, |I_4|= \dotsm =|I_m|=2 &(\zeta_3)\\
    &|I_1|=2, |I_2|=3, |I_3|=|I_4|=1, |I_5|= \dotsm =|I_m|=2 &(\zeta_4)\\
    &|I_1|=2, |I_2|=3, |I_3|=1, |I_4|= \dotsm =|I_m|=2 &(\xi_1)\\
    &|I_1|+|I_2|=4, |I_1| \leq 2, |I_3|= \dotsm =|I_m|=2 &(\xi_2)\\
    &|I_1|=2, |I_2|=3, |I_3|= \dotsm =|I_m|=2. &(\tau_1)
\end{align*}
We have seen above that $\langle X \setminus Z \rangle \subseteq J$. Now $J$ is the smallest subspace of $F$ containing $X\setminus Z$
that is invariant under taking Lie commutators with the $c_{j}$'s and $z$. In the following we go systematically through all possible scenarios where we need to add new elements to $\langle X\setminus Z\rangle$  in order to get a subspace $Z_{0}+\langle X\setminus Z\rangle $, where $Z_{0}\leq Z$, that is invariant under Lie commutators with the $c_{j}$'s. All the new elements will be relations  amongst the $\zeta$'s, $\xi$'s, and $\tau$'s. We will distinguish several cases considering the cardinality of the sets $I_1, \dots, I_m$. We will then later ensure furthermore that we obtain a subspace that is also invariant under taking Lie commutators with $z$.

\setcounter{equation}{0}
\subsection{}\label{3.2} {\bf Relations among $\zeta_1, \zeta_2, \zeta_3, \zeta_4$.}

\vspace{2mm}
\mbox{}\\
We now begin systematically adding elements to $\langle X\setminus Z\rangle$ to ensure that the resulting subspace is invariant under taking Lie commutator with $c_{k}$. For this we only need to consider elements of types $-2,-1$ and $0$. We start with a basis element $u$ from $X\setminus Z$ that is of type $-2$. We write $[u,c_{k}]$ as a linear combination of basis elements in $X$. If all the components are in $X\setminus Z$, then nothing needs to be added. We thus only need to consider the cases when some of the components are among $\zeta_{1},\zeta_{2},\zeta_{3}$ or $\zeta_{4}$. This means reducing the size of one of $I_{1}, \ldots ,I_{m}$ for some of $\zeta_{1},\ldots ,\zeta_{4}$ by one, to get a basis element $u$ of type $-2$ and then expand $[u,c_{k}]$ to get a linear combination of $\zeta_{1},\ldots ,\zeta_{4}$ modulo $J$. This may then give us  a new relator to add to $\langle X\setminus Z\rangle$. Below we exhaust all possibilities. \\ \\
\noindent {\bf Case A1: $|I_1|=2,|I_2|=3, |I_3|=0, |I_4|=1, |I_5|= \dotsm =|I_m|=2 $.}\\ \\
 Without loss of generality, we can assume that $I_1 \cup I_2 \cup \dotsm \cup I_m \cup \{k\} =\{1, \dotsc, 2m-1\}$. Recall that for the basis elements for $\mbox{Id}_{F}(z)$ we should have $1\in I_{2}$. For this reason we need to deal separately with the cases $k\geq 2$ and $k=1$. \\ \\
 Consider first the case when $k \geq 2$. Suppose $I_4=\{i\}$. Here as elsewhere in this section we will be calculating modulo $J$. We have 
\begin{eqnarray*}
0&=&[[z, c_{I_1}], [z, c_{I_2}], z, [z, c_i], [z, c_{I_5}], \dotsc, [z, c_{I_m}], c_k]\\
&=& \zeta_4(I_1, I_2, \{k\}, \{i\}, I_5, \dotsc, I_m) + \zeta_3(I_1, I_2, \varnothing, \{i,k\}, I_5, \dotsc, I_m).
\end{eqnarray*}
Therefore
\begin{equation}\tag{R1}\label{R1}
    \zeta_4(I_1, I_2, \{k\}, \{i\}, I_5, \dotsc, I_m)=-\zeta_3(I_1, I_2, \varnothing, \{i,k\}, I_5, \dotsc, I_m).
\end{equation}
This means that we need to add the relator $\zeta_4(I_1, I_2, \{k\}, \{i\}, I_5, \dotsc, I_m)+\zeta_3(I_1, I_2, \varnothing, \{i,k\}, I_5, \dotsc, I_m)$ to $\langle X\setminus Z\rangle$. \\ \\
The reader can convince himself that the case $k=1$ does not give us a new relation. \\ \\
\noindent {\bf Case A2: $|I_1|=1,|I_2|=3, |I_3|=0, |I_4|= \dotsm =|I_m|=2 $.} \\ \\
First consider the case when  $k \geq 2$ and suppose that $I_1=\{i\}$. We have

\begin{eqnarray*}
0&=&[[z, c_i], [z, c_{I_2}], z, [z, c_{I_4}], \dotsc, [z, c_{I_m}], c_k]\\
&=& [[z, c_{\{i,k\}}], [z, c_{I_2}], z, [z, c_{I_4}], \dotsc, [z, c_{I_m}]] \\
&+& [[z, c_i], [z, c_{I_2}], [z, c_k], [z, c_{I_4}], \dotsc, [z, c_{I_m}]]\\
&=&\zeta_3(\{i,k\}, I_2, \varnothing, I_4, \dotsc, I_m)+\zeta_2(\{i\}, I_2, \{k\}, I_4, \dotsc, I_m).
\end{eqnarray*}
Therefore, 
\begin{equation}\tag{R2}\label{R2}
\zeta_2(\{i\}, I_2, \{k\}, I_4, \dotsc, I_m)=-\zeta_3(\{i,k\}, I_2, \varnothing, I_4, \dotsc, I_m).
\end{equation}
Again one sees that for $k=1$, one does not get a new relation. \\ \\
\noindent {\bf Case A3: $|I_1|=|I_2|=2, |I_3|=0, |I_4|= \dotsm =|I_m|=2 $.} \\ \\
Consider the case when $k \geq 2$. We have
\begin{eqnarray*}
0&=& [[z, c_{I_1}], [z, c_{I_2}], z, [z, c_{I_4}], \dotsc, [z, c_{I_m}], c_k]\\
&=& [[z, c_{I_1}], [z, c_{{I_2}\cup \{k\}}], z, [z, c_{I_4}], \dotsc, [z, c_{I_m}]] \\
&+& [[z, c_{I_1}], [z, c_{I_2}], [z, c_k], [z, c_{I_4}], \dotsc, [z, c_{I_m}]]\\
&=&\zeta_3(I_1, I_2\cup\{k\}, \varnothing, I_4, \dotsc, I_m) + \zeta_2(I_1, I_2, \{k\}, I_4, \dotsc, I_m).
\end{eqnarray*}

Therefore, we get
\begin{equation}\tag{R3}\label{R3}
 \zeta_2(I_1, I_2, \{k\}, I_4, \dotsc, I_m)=-\zeta_3(I_1, I_2\cup\{k\}, \varnothing, I_4, \dotsc, I_m).
\end{equation}
We will handle $k=1$ later. \\ \\
\noindent {\bf Case A4: $|I_1|=2, |I_2|=1, |I_3|=1, |I_4|= \dotsm =|I_m|=2 $.} \\ \\
Let $I_3=\{i\}$. For the case $k \geq 2$ we get
\begin{eqnarray*}	
	0&=&[[z,c_{I_1}], [z,c_{\{1,k\}}], [z,c_i], [z,c_{I_4}], \dotsc, [z,c_{I_m}]]\\
	&+&[[z,c_{I_1}], [z,c_1], [z,c_{\{i,k\}}], [z,c_{I_4}], \dotsc, [z,c_{I_m}]]\\
	&=&\zeta_2(I_1, \{1,k\}, \{i\}, I_4, \dotsc, I_m)+\zeta_1(I_1, \{1\}, \{i,k\}, I_4, \dotsc, I_m)\\
	&=&-\zeta_3(I_1, \{1,i,k\}, \varnothing, I_4, \dotsc, I_m)+\zeta_1(I_1, \{1\}, \{i,k\}, I_4, \dotsc, I_m), 
\end{eqnarray*}
where the last equality follows from \ref{R3}. Therefore,
\begin{equation}\tag{R4}\label{R4}
    \zeta_1(I_1,  \{1\}, \{i,k\}, I_4, \dotsc, I_m)=\zeta_3(I_1, \{1,i,k\}, \varnothing, I_4, \dotsc, I_m).
\end{equation}
We will deal with the case $k=1$ later. \\ \\
\noindent {\bf Case A5: $|I_1|=1, |I_2|=2, |I_3|=1, |I_4|= \dotsm =|I_m|=2$.} \\ \\
Consider first the case $k\geq 2$. 
Let $I_1=\{i\}, I_3=\{j\}$ and $k \geq 2$. We have
\begin{eqnarray*}
	0&=&[[z,c_{\{i,k\}}], [z,c_{I_2}], [z,c_j], [z,c_{I_4}], \dotsc, [z,c_{I_m}]]\\
	&+&[[z,c_i], [z,c_{I_2 \cup \{k\}}], [z,c_j], [z,c_{I_4}], \dotsc, [z,c_{I_m}]]\\
	&+&[[z,c_i], [z,c_{I_2}], [z,c_{\{j,k\}}], [z,c_{I_4}], \dotsc, [z,c_{I_m}]]\\
	&=&\zeta_2(\{i,k\}, I_2, \{j\}, I_4, \dotsc, I_m)+\zeta_2(\{i\}, I_2 \cup \{k\}, \{j\}, I_4, \dotsc, I_m)\\
	&+&\zeta_1(\{i\}, I_2, \{j,k\}, I_4, \dotsc, I_m)\\
	&=&-\zeta_3(\{i,k\}, I_2 \cup \{j\}, \varnothing, I_4, \dotsc, I_m)-\zeta_3(\{i,j\}, I_2 \cup \{k\}, \varnothing, I_4, \dotsc, I_m)\\
	&+&\zeta_1(\{i\},I_2, \{j,k\}, I_4, \dotsc, I_m), 
\end{eqnarray*}
where the last equality follows from R2 and \ref{R3}. Therefore,
\begin{align}\tag{R5}\label{R5}
    \zeta_1(\{i\},  I_2, \{j,k\}, I_4, \dotsc, I_m)&= \zeta_3(\{i,k\}, I_2 \cup \{j\}, \varnothing, I_4, \dotsc, I_m)\\
    \nonumber &+ \zeta_3(\{i,j\}, I_2 \cup \{k\}, \varnothing, I_4, \dotsc, I_m).
\end{align}
We deal with the case $k=1$ later. \\ \\
\noindent {\bf Case A6: $|I_1|=0, |I_2|=3, |I_3|=1, |I_4|= \dotsm =|I_m|=2 $.} \\ \\
We first look at the case $k\geq 2$. 
Let $|I_3|=\{i\}$. We have 

\begin{eqnarray*}	
	0&=&[[z,c_k], [z, c_{I_2}], [z,c_i], [z,c_{I_4}], \dotsc, [z,c_{I_m}]]\\
	&+&[z, [z, c_{I_2}], [z,c_{\{i,k\}}], [z,c_{I_4}], \dotsc, [z,c_{I_m}]]\\
	&=&\zeta_2(\{k\}, I_2, \{i\}, I_4, \dotsc,I_m)+\zeta_1(\varnothing, I_2, \{i,k\}, I_4, \dotsc, I_m)\\
	&=&-\zeta_3(\{k,i\}, I_2, \varnothing, I_4, \dotsc, I_m)+\zeta_1(\varnothing, I_2, \{i,k\}, I_4, \dotsc, I_m).
\end{eqnarray*}
Therefore,
\begin{equation}\tag{R6}\label{R6}
    \zeta_1(\varnothing, I_2, \{i,k\}, I_4, \dotsc, I_m)=\zeta_3(\{i,k\}, I_2, \varnothing, I_4, \dotsc, I_m).
\end{equation}
Again the case $k=1$ does not add anything new. \\ \\
Notice that it follows from the relations above that modulo $J$ we can write the basis elements of types $\zeta_1, \zeta_2$ and $\zeta_4$ as a linear combination of basis elements of type $\zeta_3$. Notice also that from (R6) and the fact that $\mbox{Id}_{F}(z)$ is metabelian, it follows that 
\begin{eqnarray*}
\zeta_3(I_1,I_2, \varnothing, I_4, I_5, \dotsc , I_m) & = & \zeta_1(\varnothing, I_2,I_1,I_4, I_5, \dotsc, I_m) \\
\mbox{} & = & \zeta_1(\varnothing, I_2,I_4,I_1,I_5, \dotsc, I_m) \\
 \mbox{} & = & \zeta_3(I_4,I_2,\varnothing ,I_1, I_5, \dotsc ,I_m).
\end{eqnarray*}
Therefore we get

\begin{align}\tag{R7}\label{R7}
    \zeta_3(I_1,I_2,\varnothing, I_4, I_5, \dotsc, I_m) = & \zeta_3(I_4, I_2, \varnothing, I_1, I_5, \dotsc, I_m).
\end{align}

\noindent {\bf Case A7: $|I_1|=|I_2|=1, |I_3|= \dotsm =|I_m|=2 $.} \\ \\
Let $I_1=\{i\}, I_3=\{j,r\}$ and consider the case where $k \geq 2$. We have
\begin{eqnarray*}
	0&=&[[z,c_{\{i,k\}}], [z,c_1], [z,c_{\{j,r\}}], [z,c_{I_4}], \dotsc, [z,c_{I_m}]]\\
	&+&[[z,c_i], [z,c_{\{1,k\}}], [z,c_{\{j,r\}}], [z,c_{I_4}], \dotsc, [z,c_{I_m}]]\\
	&=&\zeta_1(\{i,k\}, \{1\}, \{j,r\}, I_4, \dotsc, I_m)+\zeta_1(\{i\}, \{1,k\}, \{j,r\}, I_4, \dotsc, I_m\\
	&=&\zeta_3(\{i,k\}, \{1,j,r\}, \varnothing, I_4, \dotsc, I_m)+\zeta_3(\{i,j\}, \{1,k,r\}, \varnothing, I_4, \dotsc, I_m)\\
	&+&\zeta_3(\{i,r\}, \{1,j,k\}, \varnothing, I_4, \dotsc, I_m), \text{ (using \ref{R4} and \ref{R5})}.
\end{eqnarray*}

Therefore,
\begin{align}\tag{R8}\label{R8}
    \zeta_3(\{i,k\},  \{1,j,r\}, \varnothing, I_4, \dotsc, I_m)=& -\zeta_3(\{i,j\}, \{1,k,r\}, \varnothing, I_4, \dotsc, I_m)\\
    \nonumber & - \zeta_3(\{i,r\}, \{1,j,k\}, \varnothing, I_4, \dotsc, I_m).
\end{align}
For $k=1$ one can check that there is not a new relation. \\ \\
\noindent {\bf Case A8: $|I_1|=0, |I_2|=2, |I_3|= \dotsm =|I_m|=2 $.} \\ \\
Let $I_2=\{1,i\}, I_3=\{j,r\}$ and $k \geq 2$. We have
\begin{eqnarray*}
	0&=&[z, [z,c_{\{1,i,k\}}], [z,c_{I_3}], \dotsc, [z,c_m]]+[[z,c_k], [z,c_{\{1,i\}}], [z,c_{I_3}], \dotsc, [z,c_{I_m}]]\\
	&=&\zeta_1(\varnothing, \{1,i,k\}, I_3, \dotsc, I_m)+\zeta_1(\{k\}, \{1,i\}, I_3, \dotsc, I_m)\\
	&=&\zeta_3(\{j,r\}, \{1,i,k\}, \varnothing, I_4, \dotsc, I_m)+\zeta_3(\{k,j\}, \{1,i,r\},  \varnothing, I_4, \dotsc, I_m)\\
	&+&\zeta_3(\{k,r\}, \{1,i,j\},  \varnothing, I_4, \dotsc, I_m).
\end{eqnarray*}
Modulo relation \ref{R8}, we get
\begin{equation}\tag{R9}\label{R9}
    \zeta_3(\{j,r\},  \{1,i,k\}, \varnothing, I_4, \dotsc, I_m)=\zeta_3(\{i,k\}, \{1,j,r\}, \varnothing, I_4, \dotsc, I_m).
\end{equation}
The reader can check that here fas well as for A3, A4  and A5 with $k=1$ one does not obtain any new relations.  The same is true for the following remaining cases. \\ \\
\noindent {\bf Case A9. $|I_1|=|I_2|=2, |I_3|=|I_4|=1, |I_5|= \dotsm =|I_m|=2 $.}\label{case2.2.6} \\
\noindent{\bf Case A10: $|I_1|=2,|I_2|=3, |I_3|=|I_4|=|I_5|=1, |I_6|= \dotsm =|I_m|=2 $.} \\ 
\noindent {\bf Case A11: $|I_1|=1,|I_2|=3, |I_3|=|I_4|=1, |I_5|= \dotsm =|I_m|=2 $.} \\
\noindent {\bf Case A12: $|I_{1}|=2,|I_{2}|=0,|I_{3}|=\dotsm =|I_{m}|=2$.} \\ \\
To summarize we get the following set of equivalent relations for elements of type $-1$ modulo $J$.

\begin{align*}
    & \zeta_4(I_1, I_2, \{i\}, \{j\}, I_5, \dotsc, I_m)=-\zeta_3(I_1, I_2, \varnothing, \{i,j\}, I_5, \dotsc, I_m) &(E1) \\
    &\zeta_2(I_1, I_2, \{i\}, I_4, \dotsc, I_m)=-\zeta_3(I_1, I_2 \cup \{i\}, \varnothing, I_4, \dotsc, I_m) \mbox{ }(|I_{1}|=|I_{2}|=2) &(E2)\\
    &\zeta_2(\{i\}, I_2, \{j\}, I_4, \dotsc, I_m)=-\zeta_3(\{i,j\}, I_2, \varnothing, I_4, \dotsc, I_m) &(E3)\\
    &\zeta_1(\varnothing, I_2, I_3, I_4, \dotsc, I_m)=\zeta_3(I_3, I_2, \varnothing, I_4, \dotsc, I_m) &(E4)\\
    &\zeta_1(I_1, \{1\}, \{i,j\}, I_4, \dotsc, I_m)=\zeta_3(I_1, \{1,i,j\}, \varnothing, I_4, \dotsc, I_m) &(E5)\\
    &\zeta_1(\{i\}, I_2, \{j,k\}, I_4, \dotsc, I_m)=\zeta_3(\{i,j\}, I_2 \cup \{k\}, \varnothing, I_4, \dotsc, I_m)\\
    &\nonumber +\zeta_3(\{i,k\}, I_2 \cup \{j\}, \varnothing, I_4, \dotsc, I_m) &(E6)\\
    &\zeta_3(I_1, I_2, \varnothing, I_4, I_5, \dotsc, I_m)=\zeta_3(I_4, I_2, \varnothing, I_1, I_5, \dotsc, I_m) &(E7)\\
    &\zeta_3(I_1, I_2 \cup \{1\}, \varnothing, I_4, \dotsc, I_m)=\zeta_3(I_2, I_1\cup \{1\}, \varnothing, I_4, \dotsc, I_m) &(E8)\\
    &\zeta_3(\{i,k\}, \{1,j,r\}, \varnothing, I_4, \dotsc, I_m)+\zeta_3(\{i,j\}, \{1,k,r\}, \varnothing, I_4, \dotsc, I_m)\\
    &+\nonumber \zeta_3(\{i,r\}, \{1,j,k\}, \varnothing, I_4, \dotsc, I_m)=0. &(E9)
\end{align*}
Notice that relations E1-E6 tell us that, modulo $J$, all the basis elements of type $-1$ in $X$ can be written as linear combinations of elements of type $\zeta_{3}$. Then E7 and E8 imply that 
                $$\zeta_3(I_{1},I_{2},I_{4},\ldots ,I_{m})=\zeta_{3}(I_{1},I_{2}\cup \{1\},\varnothing,I_{4},\ldots ,I_{m})$$
is symmetric in $I_{1},I_{2},I_{4},\ldots ,I_{m}$. Apart from this we have one extra relation E9.

\subsection{Relations among $\xi_1, \xi_2$.}\label{3.3}
\mbox{}\\ \\
We continue the analysis and deal here with $[u,c_{k}]$ were $u$ is of type $-1$. On readily sees that if $u\in X\setminus Z$, then $[u,c_{k}]\in \langle X\setminus Z\rangle$ and no new relators needed. We therefore only need to consider $u$ where $u$ is one of the extra relators that we obtained in Section 3.2. \\ \\
\noindent {\bf B1. Consequences of E1.}
For $k \geq 2$ we have
$$
    [\zeta_4(I_1, I_2, \{i\}, \{j\}, I_5, \dotsc, I_m), c_k]=-[\zeta_3(I_1, I_2, \varnothing, \{i,j\}, I_5, \dotsc, I_m), c_k]
$$
which implies that
\begin{align*}
   \xi_2(I_1, I_2, \{j\}, \{i,k\}, I_5, \dotsc, I_m)=&- \xi_2(I_1, I_2, \{i\}, \{j,k\}, I_5, \dotsc, I_m)\\
    &-\xi_2(I_1, I_2, \{k\}, \{i,j\}, I_5, \dotsc, I_m).
\end{align*}

Therefore,
\begin{align}\tag{F1}\label{F1}
    \xi_2(I_1, I_2, \{j\}, \{i,k\}, I_5, \dotsc, I_m)=&- \xi_2(I_1, I_2, \{i\}, \{j,k\}, I_5, \dotsc, I_m)\\
    &\nonumber -\xi_2(I_1, I_2, \{k\}, \{i,j\}, I_5, \dotsc, I_m).
\end{align}
For $k=1$ one does not get a new relation. \\ \\
\noindent {\bf B2. Consequences of E8.}
For the case $k \geq 2$ we get
$$[\zeta_3(I_1,I_2 \cup \{1\}, \varnothing, I_4, \dotsc, I_m), c_k]=[\zeta_3(I_2, I_1,\{1\}, \varnothing, I_4, \dotsc, I_m), c_k].$$

Therefore,
\begin{equation}\tag{F2}\label{F2}
    \xi_2(I_1, I_2 \cup \{1\}, \{k\}, I_4, \dotsc, I_m)=\xi_2(I_2, I_1 \cup \{1\}, \{k\}, I_4, \dotsc, I_m).
\end{equation}
Again for $k=1$ we do not get any new relation. \\ \\
\noindent {\bf B3. Consequences of E7.}
For the case $k \geq 2$ we get
$$[\zeta_3(I_1,I_2, \varnothing, I_4,I_5, \dotsc, I_m), c_k]=[\zeta_3(I_4, I_2, \varnothing, I_1,I_5, \dotsc, I_m), c_k].$$

Therefore,
\begin{equation}\tag{F3}\label{F3}
    \xi_2(I_1, I_2, \{k\}, I_4, I_5, \dotsc, I_m)=\xi_2(I_4, I_2, \{k\}, I_1, I_5, \dotsc, I_m).
\end{equation}
For $k=1$ we do not get anything new. \\ \\
\noindent {\bf B4. Consequences of E9.}
Commuting $E_9$ with $a_s$, for the case $s \geq 2$ we get 
\begin{align}\tag{F4}\label{F4}
   \xi_2(\{i,k\}, \{1,j,r\}, \{s\}, I_4, \dotsc, I_m)=&-\xi_2(\{i,j\}, \{1,k,r\}, \{s\}, I_4, \dotsc, I_m)\\
   &\nonumber -\xi_2(\{i,r\}, \{1,j,k\}, \{s\}, I_4, \dotsc, I_m).
\end{align}
For $k=1$ we get nothing new.  \\ \\
\noindent {\bf B5. Consequences of E4.}
For the case $k \geq 2$ we get
$$
[\zeta_1(\varnothing,I_2, I_3, I_4, \dotsc, I_m), c_k]=[\zeta_3(I_3, I_2, \{k\}, I_4, \dotsc, I_m), c_k].
$$
As a consequence we obtain that
$$  
\xi_2(\{k\}, I_2, I_3,I_4, \dotsc, I_m)=\xi_2(I_3, I_2, \{k\}, I_4, \dotsc, I_m).
$$
Therefore,
\begin{equation}\tag{F5}\label{F5}
    \xi_1(\{i\}, I_2, \{j,k\}, I_4, \dotsc, I_m)=\xi_2(\{j,k\}, I_2, \{i\}, I_4, \dotsc, I_m).
\end{equation}
Again there is no relation for $k=1$. \\ \\
\noindent {\bf B6. Consequences of E5.}
For the case $k\geq 2$ we have
$$
    [\zeta_1(I_1, \{1\},\{i,j\}, I_4, \dotsc, I_m), c_k]=[\zeta_3(I_1, \{1,i,j\}, \varnothing, I_4, \dotsc, I_m), c_k].
$$
Then
$$
\xi_1(I_1, \{1,k\}, \{i,j\}, I_4, \dotsc, I_m)=\xi_2(I_1, \{1,i,j\}, \{k\}, I_4, \dotsc, I_m).
$$

Therefore,
\begin{equation}\tag{F6}\label{F6}
    \xi_1(I_1, \{1,k\}, \{i,j\}, I_4, \dotsc, I_m)=\xi_2(I_1, \{1,i,j\}, \{k\}, I_4, \dotsc, I_m).
\end{equation}
For $k=1$ we do not get anything new and the same is true when consider consequences of E2,E3 and E6. From this analysis
we get the following set of equivalent relations. 
\begin{align*}
&\xi_1(\{i\}, I_2, \{j,k\}, I_4, \dotsc, I_m)=\xi_2(\{j,k\}, I_2, \{i\}, I_4, \dotsc, I_m) &(G1)\\
& \xi_1(I_1, \{1,k\}, \{i,j\}, I_4, \dotsc, I_m)=\xi_2(I_1, \{1,i,j\}, \{k\}, I_4, \dotsc, I_m) &(G2)\\
& \xi_2(I_1, I_2 \cup \{1\}, \{k\}, I_4, \dotsc, I_m)=\xi_2(I_2, I_1 \cup \{1\}, \{k\}, I_4, \dotsc, I_m) &(G3)\\
&  \xi_2(I_1, I_2, \{k\}, I_4, I_5, \dotsc, I_m)=\xi_2(I_4, I_2, \{k\}, I_1, I_5, \dotsc, I_m) &(G4)\\
    &\xi_2(\{i,k\}, \{1,j,r\}, \{s\}, I_4, \dotsc, I_m)+\xi_2(\{i,j\}, \{1,k,r\}, \{s\}, I_4, \dotsc, I_m)\\
    &+\xi_2(\{i,r\}, \{1,j,k\}, \{s\}, I_4, \dotsc, I_m)=0 &(G5)\\
    &\xi_2(I_1, \{1,i,k\}, \{j\}, I_4, \dotsc, I_m)+\xi_2(I_1, \{1,k,j\}, \{i\}, I_4, \dotsc, I_m)\\
    &+\xi_2(I_1, \{1,i,j\}, \{k\}, I_4, \dotsc, I_m)=0. &(G6)
\end{align*}

\subsection{Relations in $\tau_1$.}\label{3.4}
\mbox{}\\ \\
Here we deal with $[u,c_{k}]$ where $u$ is of type $0$. Again we only get a new relation when $u$ is one of the relators we
obtained in Section 3.3. \\ \\
\noindent {\bf C1. Consequences of G3.}
For $k \geq 2$ we have 
$$
[\xi_2(I_1, I_2 \cup \{1\}, \{i\}, I_4, \dotsc, I_m), c_k]=[\xi_2(I_2, I_1 \cup \{1\}, \{i\},I_4, \dotsc, I_m), c_k].
$$
which gives
$$\tau_1(I_1, I_2 \cup \{1\}, \{i,k\}, I_4, \dotsc, I_m)=\tau_1(I_2, I_1 \cup \{1\}, \{i,k\}, I_4, \dotsc, I_m).
$$
Therefore, we get
\begin{equation}\tag{H1}\label{H1}
    \tau_1(I_1, I_2 \cup \{1\}, I_3, \dotsc, I_m)=\tau_1(I_2, I_1 \cup \{1\}, I_3, \dotsc, I_m).
\end{equation}
For $k=1$ we get nothing new.  \\ \\
\noindent {\bf C2. Consequences of G4.}
For $k \geq 2$ we have 
\begin{equation}\tag{H2}\label{H2}
    \tau_1(I_1, I_2, I_3, I_4, \dotsc, I_m)=\tau_1(I_4, I_2, I_3, I_1, \dotsc, I_m).
\end{equation}
For $k=1$ we get nothing new. \\ \\
\noindent {\bf C3. Consequences of G5.}
For $k \geq 2$ we have 
\begin{align}\tag{H3}\label{H3}
    \tau_1(\{i,k\}, \{1,j,r\}, I_3, \dotsc, I_m)=&-\tau_1(\{i,j\}, \{1,k,r\}, I_3, \dotsc, I_m)\\
    &\nonumber-\tau_1(\{i,r\}, \{1,k,j\}, I_3, \dotsc, I_m).
\end{align}
Again for $k=1$ we get nothing new and the same is true for the consequence of G1, G2 and G6. 
Thus the only extra relators we get are H1, H2 and H3. We will now see that if we add to $\langle X\setminus Z\rangle$ the subspace generated by the extra relators we have obtained in Sections 3.2, 3.3 and 3.4, then we get an ideal in $F$. 
\begin{lem}
The subspace generated by $X\setminus Z$ and the relators given in E1-E9, G1-G6 and H1-H3 is an ideal in $F$ and thus equal to $J$.
\end{lem}
\begin{proof}
Let $V$ be this subspace. We found the relations E1-E9, G1-G6 and H1-H3 by systematically ensuring that we obtained a subspace that is invariant under taking Lie commutators with the $c_{j}$'s. Notice that if we take a basis element $u=[[c,c_{i_{1}}],[z,c_{i_{2}}],\ldots ,[z,c_{I_{m}}]]$ from $X$ where $[u,z]\not\in V$ then we must have $|I_{1}|,|I_{3}|,|I_{4}|,\ldots ,|I_{m}|\leq 2$ and $|I_{2}|\leq 3$. Thus $u$ has type at most $1$. On the other hand if the type of $u$ is less than $1$, then $[u,z]$ has type less than $-1$ and is thus in $V$. This means that we only need to consider $u\in V$ that is linear combination of elements of type $\tau_{1}$. 
In other words we have to  commute each relation H1, H2, H3 with $z$ and check that the resulting element is in $V$. 
If we commute relation H1 with $z$, we get
$$
[\tau_1(I_1, I_2 \cup \{1\}, I_3, \dotsc, I_m),z]=[\tau_1(I_2, I_1 \cup \{1\}, I_3, \dotsc, I_m),z]
$$
which implies that
$$\zeta_3(I_1, I_2 \cup \{1\}, \varnothing, I_4, \dotsc, I_m)=\zeta_3(I_2, I_1 \cup \{1\}, \varnothing, I_4, \dotsc, I_m),
$$
which is relation E8 from Section \ref{3.2}. Using similar arguments we commute relation H2 with $z$, in which case we have
$$[\tau_1(I_1, I_2, I_3, I_4, \dotsc, I_m),z]=[\tau_1(I_4, I_2, I_3, I_1, \dotsc, I_m), z]$$
which implies 
$$\zeta_3(I_1, I_2, \varnothing, I_4, I_5, \dotsc, I_m)=\zeta_3(I_4, I_2, \varnothing, I_1, I_5, \dotsc, I_m)$$
which is relation E7 from Section \ref{3.2}. Lastly, commuting relation H3 with $z$ gives
\begin{eqnarray*}
    [\tau_1(\{i,k\}, \{1,j,r\}, I_3, \dotsc, I_m),z]&=&-[\tau_1(\{i,j\}, \{1,k,r\}, I_3, \dotsc, I_m),z]\\
    &&-[\tau_1(\{i,r\}, \{1,k,j\}, I_3, \dotsc, I_m),z]
\end{eqnarray*}
which implies 
\begin{eqnarray*}
    \zeta_3(\{i,k\}, \{1,j,r\}, \varnothing, I_4, \dotsc, I_m)&=&-\zeta_3(\{i,j\}, \{1,k,r\}, \varnothing, I_4, \dotsc, I_m)\\
    &&- \zeta_3(\{i,r\}, \{1,j,k\}, \varnothing, I_4, \dotsc, I_m)
\end{eqnarray*}
which is relation E9. This completes the proof. 

%
\end{proof}

\section{The ideal $\Id_L(z)$}
\mbox{}\\
Let $W=W(m,1,\overset{2m+1}{\dots},1)$  be the subspace of $F$ generated by all 
$$
E(I_1, \dots, I_m)=[[z,c_{I_1}], [z,c_{\{1\} \cup I_2}], [z,c_{I_3}], \dots,  [z,c_{I_m}]]
$$
with $I_1 \cup \dots \cup I_m=\{2, \dots, 2m+1\}$ and $|I_1|=\dots=|I_m|=2$. We know that these elements are linearly independent and we have also seen that 
$$
N=N(m, 1,\overset{2m+1}{\dots}, 1)=W(m, 1, \overset{2m+1}{\dots}, 1) \cap J
$$
is generated by two sets of relations
\begin{align}
\tag{$\mathcal{R}_1$}\label{R1}
 E(I_{\sigma(1)}, \dots, I_{\sigma(m)})&-E(I_1,\dots,I_m), \mbox{ where $\sigma \in \Sym \{1, \dotsc, m\}$ }\\ 
\tag{$\mathcal{R}_2$} \label{R2} E(\{i_1,i_2\},\{j_1,j_2\},I_3,\dots,I_m)=&-E(\{i_1,j_1\},\{i_2,j_2\},I_3,\dots,I_m)\\
\nonumber
&-E(\{i_1,j_2\},\{i_2,j_1\},I_3,\dots,I_m).
\end{align}
Our aim is to show that $\frac{W}{N} \neq 0$, for all $m \geq 1$. Then it will follow that $L=\frac{F}{J}$ has multihomogenous elements of arbitrary weight $m$ in $z$ that are non-zero. As a consequence, we see that $\Id_L(z)$ is not nilpotent. \\ \\
From now on, we work modulo (\ref{R1}) and thus the order of $I_1, \dots, I_m$ does not matter for the value of $E(I_1, \dots, I_m)$. From now on we write $E(\{I_1, \dots, I_m\})$ for $E(I_1, \dots, I_m)$ modulo (\ref{R1}). In this way we get an element for each partition $\{I_1, \dots, I_m\}$ and these are a basis for $W$ modulo (\ref{R1}). \\ \\
{\bf Definition}.  We say that a basis element $E(\{I_1, \dots, I_m\})$ has \textit{norm} $k$ if there are exactly $k$ of the $I_j$'s, where $I_j \subseteq \{2, \dots, m+1\}$. Notice that we then also have exactly $k$ of the $I_j$'s, where $I_j \subseteq \{m+2, \dots, 2m+1\}$. \\ \\
Suppose that $k \geq 1$. We can assume $I_1, I_3, \dots,  I_{2k-1} \subseteq \{2, \dots, m+1\}$ and $I_2, I_4, \dots,  I_{2k} \subseteq \{m+2, \dots, 2m+1\}$. If $I_1=\{i_1, i_2\}$ and $I_2=\{j_1, j_2\}$, then modulo ($\mathcal{R}_2$) we have
\begin{align*}
E(\{\{i_1,i_2\},\{j_1,j_2\},I_3,\dots,I_m\})=&-E(\{\{i_1,j_1\},\{i_2,j_2\},I_3,\dots,I_m\})\\
&-E(\{\{i_1,j_2\},\{i_2,j_1\},I_3,\dots,I_m\}).  
\end{align*}
Notice that the terms on the right hand side have norm $k-1$. We will refer to this as a 1-step decomposition of $E$ through the pair $(I_1, I_2)$ and we denote it $\mathcal{D}_{\{(I_1,I_2)\}}(E(\{I_1, \dots, I_m\}))$. 
Suppose now $k\geq 2$. We can continue and apply the 1-step decomposition to the two terms through the pair $(I_3, I_4)$ to get a 2-step decomposition of $E(I_1, \dots, I_m)$ through $\{(I_1, I_2),(I_3,I_4)\}$ that we denote 
$$
\mathcal{D}_{\{(I_1,I_2),(I_3,I_4)\}}(E(\{I_1, \dots, I_m\})).
$$
We can continue in this manner through the pairs $(I_1, I_2), \dots, (I_{2k-1}, I_{2k})$ to get a decomposition into $2^k$ terms of norm 0, denoted 
$$
\mathcal{D}_{\{(I_1,I_2),\dots, (I_{2k-1},I_{2k})\}}(E(\{I_1, \dots, I_m\})).
$$
{\bf Remark}. To say that an element $E(\{I_{1},I_{2},\ldots ,I_{m}\})$ has norm $0$ is to say that $\{I_{1},\ldots ,I_{m}\}=\{(2,\sigma(m+2)),(3,\sigma(m+3)),\ldots ,(m+1,\sigma(2m+1))\}$ for some permutation of $\sigma$ of $m+2,\ldots ,2m+1$. We have seen that, modulo ($\mathcal{R}_{1}$) and ($\mathcal{R}_{2}$), $W$ is generated by basis elements of norm $0$. Let $W_{0}$ be the subspace of $W$ generated by these. Thus $W/W\cap J\cong (W_{0}+J)/J\cong W_{0}/W_{0}\cap J$.    \\ \\
Notice that the decomposition $\mathcal{D}_{\{(I_1,I_2),\dots, (I_{2k-1},I_{2k})\}}(E(\{I_1, \dots, I_m\}))$, does not depend on the order of the pairs in $\{(I_1,I_2),\dots, (I_{2k-1},I_{2k})\}$, but it may depend on what pairings are chosen. For each $E(\{I_1, \dots, I_m\})$ of norm  $k \geq 2$, pick one of the decompositions according to some pairing and denote this $E^0(\{I_1, \dots, I_m\})$. Thus for every element $E(\{I_1, \dots, I_m\})$ of norm $k \geq 2$, there is a defining relation
$$
E(\{I_1, \dots, I_m\})=E^0(\{I_1, \dots, I_m\})
$$
modulo (\ref{R2}), where $E^0(\{I_1, \dots, I_m\})$ is a linear combination of terms of norm 0. Using these defining relations, all relations in (\ref{R2}) become relations in $W_0$, the subspace of $W$ generated by $E(\{I_1, \dots, I_m\})$ of norm  0. We next want to understand what these relations in $W_0$ are. Let $(\mathcal{R}_2)_2$ be the collections of all relations in (\ref{R2}) where none of the $I_5, \dots, I_m$ is contained in $\{2, \dots, m+1\}$ or $\{m+2, \dots, 2m+1\}$. Next lemma simplifies our task.

\begin{lem}\label{relationmoduloR2}
Let $E=E(\{I_1, \dots, I_m\})$ be a basis element of norm  $k$ where $I_1, I_3, \dots,  I_{2k-1} \subseteq \{2, \dots, m+1\}$ and $I_2, I_4, \dots,  I_{2k} \subseteq \{m+2, \dots, 2m+1\}$ modulo  $(\mathcal{R}_2)_2$. We have
$$
\mathcal{D}_{\{(I_1,I_2),\dots, (I_{2k-1},I_{2k})\}}(E)=\mathcal{D}_{\{(I_1,I_{\sigma(2)}),\dots, (I_{2k-1},I_{\sigma(2k)})\}}(E)
$$
for any $\sigma \in \Sym\{2,4,\dots, 2k\}$.
\end{lem}

\begin{proof}
We prove the result by induction on $k$. If $k=1$ then it is clear, since in this case $\sigma$ is the identity. For $k=2$ this is a direct consequence of $({\mathcal R}_{2})_{2}$. Now suppose that $k\geq 3$ and that the result is true for smaller values of $k$.  Without loss of generality, we can suppose $\sigma(2)=4$. Using the induction hypothesis twice we have
\begin{align*}
&\mathcal{D}_{\{(I_1,I_4),(I_3,I_{\sigma(4)}),(I_5,I_{\sigma(6)}),\dots, (I_{2k-1},I_{\sigma(2k)})\}}(E)\\
&=\mathcal{D}_{\{(I_1,I_4),(I_3,I_2),(I_5,I_6),\dots, (I_{2k-1},I_{2k})\}}(E)\\
&=\mathcal{D}_{\{(I_1,I_2),(I_3,I_4),(I_5,I_6),\dots, (I_{2k-1},I_{2k})\}}(E),
\end{align*}
as required.
\end{proof}
Now take any relation in (\ref{R2}). Without loss of generality, we can assume (after reordering $I_{1},\ldots ,I_{m}$) that half of the elements of $I_{5}\cup\dots \cup I_{m}$ are in $\{2,\dots ,m+1\}$ and half in $\{m+2,\dots ,2m+1\}$. We are using here
the fact that  the size of $(I_1\cup I_2) \cap \{2, \dots, m+1\}$ is between $0$ and $4$). We thus have now a relation of the form:
\begin{align}\label{relation*}
E(\{\{i_1,i_2\},\{j_1,j_2\},I_3,\dots,I_m\})=&-E(\{\{i_1,j_1\},\{i_2,j_2\},I_3,\dots,I_m\})\\
\nonumber &-E(\{\{i_1,j_2\},\{i_2,j_1\},I_3,\dots,I_m\}).  
\end{align}
Suppose exactly $k$ of $I_5, \dots, I_m$ are subsets of $\{2, \dots, m+1\}$. Without loss of generality we can suppose that 
$$
I_5, I_7, \dots,  I_{2k+3} \subseteq \{2, \dots, m+1\} \mbox{ and } I_6, I_8, \dots,  I_{2k+4} \subseteq \{m+2, \dots, 2m+1\}.
$$
By Lemma \ref{relationmoduloR2}, we know that modulo $(\mathcal{R}_2)_2$, relation \eqref{relation*} is equivalent to a sum of relations of the type
\begin{eqnarray}
E(\{\{i_1,i_2\},\{j_1,j_2\},I_3, I_4, J_5,\dots,J_m\}) & = & -E(\{\{i_1,j_1\},\{i_2,j_1\},I_3, I_4, J_5,\dots,J_m\}) \nonumber \\ 
      & & -E(\{\{i_1,j_2\},\{i_2,j_1\},I_3, I_4, J_5,\dots,J_m\}).  
\end{eqnarray}
where no $J_5, \dots, J_m$ is a subset of $\{2, \dots, m+1\}$ or $\{m+2, \dots, 2m+1\}$. Notice that  all the relations (2) are in $(\mathcal{R}_2)_2$.\\ \\
The conclusion is that all the relations in $W_0$ we are looking for, will be consequence of $(\mathcal{R}_2)_2$.  Let $J_5, \dots, J_m$ be fixed and consider the collection of all the relations in $(\mathcal{R}_2)_2$ where $I_5=J_5 \dots I_m=J_m$.  \\ \\
Let $\bar{E}(\{I_1, I_2, I_3, I_4\})=E(\{I_1,I_2, I_3, I_4, J_5, \dotsc, J_m\})$. Suppose 
$\{2,\dotsc, m+1\} \setminus (J_5 \cup \dotsm \cup J_m)$$=\{i_1, i_2, i_3, i_4, j_1, j_2, j_3, j_4\},$ where $i_1, i_2, i_3, i_4 \in \{2, \dotsc, m+1\}$ and $j_1, j_2, j_3, j_4 \in \{m+2, \dotsc, 2m+1\}$. We now go systematically through all possible types of $({\mathcal R}_{2})_{2}$ relations. \\ \\
{\bf Case 1}: $I_3=\{i_3,j_3\}, I_4=\{i_4, j_4\}$.  We have 
\begin{eqnarray*}
    \bar{E}(\{\{i_1, i_2\}, \{j_1, j_2\}, \{i_3, j_3\}, \{i_4, j_4\}\}) & = &-\bar{E}(\{\{i_1, j_1\}, \{i_2, j_2\}, \{i_3, j_3\}, \{i_4, j_4\}\})\\
    & & -\bar{E}(\{\{i_1, j_2\}, \{i_2, j_1\}, \{i_3, j_3\}, \{i_4, j_4\}\}).
\end{eqnarray*}
 These are defining relations. \\ \\
 {\bf Case 2}: $I_3=\{i_3,i_4\}, I_4=\{j_3, j_4\}$. We have
 \begin{eqnarray*}
     \bar{E}(\{\{i_1, i_2\}, \{j_1, j_2\}, \{i_3, i_4\}, \{j_3, j_4\}\}) & = & -\bar{E}(\{\{i_1, j_1\}, \{i_2, j_2\}, \{i_3, i_4\}, \{j_3, j_4\}\}) \\
     & &    -\bar{E}(\{\{i_1, j_2\}, \{i_2, j_1\}, \{i_3, i_4\}, \{j_3, j_4\}\})\\
     &=&\bar{E}(\{\{i_1, j_1\}, \{i_2, j_2\}, \{i_3, j_3\}, \{i_4, j_4\}\}) \\ 
      & & +\bar{E}(\{\{i_1, j_1\}, \{i_2, j_2\}, \{i_3, j_4\}, \{i_4, j_3\}\})\\
     & &+\bar{E}(\{\{i_1, j_2\}, \{i_2, j_1\}, \{i_3, j_3\}, \{i_4, j_4\}\}) \\
    & & +\bar{E}(\{\{i_1, j_2\}, \{i_2, j_1\}, \{i_3, j_4\}, \{i_4, j_3\}\}).
 \end{eqnarray*}
\normalsize
If we instead pair $\{i_1, i_2\}$ and $\{j_3, j_4\}$ we get 
\begin{eqnarray*}
    \bar{E}(\{\{i_1, i_2\}, \{j_1, j_2\}, \{i_3, i_4\}, \{j_3, j_4\}\}) & = & -\bar{E}(\{\{i_1, j_3\}, \{j_1, j_2\}, \{i_3, i_4\}, \{i_2, j_4\}\})\\
   & &-\bar{E}(\{\{i_1, j_4\}, \{j_1, j_2\}, \{i_3, i_4\}, \{i_2, j_4\}\})\\
   & =&\bar{E}(\{\{i_1, j_3\}, \{i_2, j_4\}, \{i_3, j_1\}, \{i_4, j_2\}\}) \\
  & & +\bar{E}(\{\{i_1, j_3\}, \{i_2, j_4\}, \{i_3, j_2\}, \{i_4, j_1\}\})\\
   & &\bar{E}(\{\{i_1, j_4\}, \{i_2, j_3\}, \{i_3, j_1\}, \{i_4, j_2\}\}) \\
   & & +\bar{E}(\{\{i_1, j_4\}, \{i_2, j_3\}, \{i_3, j_2\}, \{i_4, j_1\}\}).
\end{eqnarray*}
\normalsize
From these two pairings we get 
\begin{align*}
& \mathlarger{\mathlarger{\sum}}_{\tiny \sigma \in \mbox{Sym}\{1,2\}, \tau \in \mbox{Sym}\{3,4\}} \bar{E}(\{\{i_1, j_{\sigma(1)}\}, \{i_2, j_{\sigma(2)}\}, \{i_3, j_{\tau(3)}\}, \{i_4, j_{\tau(4)}\}\}) \\
= & \mathlarger{\mathlarger{\sum}}_{\tiny  \sigma \in \mbox{Sym}\{1,2\}, \tau \in \mbox{Sym}\{3,4\}} \bar{E}(\{\{i_1, j_{\tau(3)}\}, \{i_2, j_{\tau(4)}\}, \{i_3, j_{\sigma(1)}\}, \{i_4, j_{\sigma(2)}\}\}).
\end{align*}
{\bf Case 3}: $I_3=\{j_2,j_3\}, I_4=\{i_4, j_4\}$.  We have
\begin{eqnarray*}
    0 & =&\bar{E}(\{\{i_1, i_2\}, \{i_3, j_1\}, \{j_2, j_3\}, \{i_4, j_4\}\})+\bar{E}(\{\{i_1, i_3\}, \{i_2, j_1\}, \{j_2, j_3\}, \{i_4, j_4\}\})\\
    & & +\bar{E}(\{\{i_2, i_3\}, \{i_1, j_1\}, \{j_2, j_3\}, \{i_4, j_4\}\})\\
   & =&-\bar{E}(\{\{i_1, j_2\}, \{i_2, j_1\}, \{i_3, j_3\}, \{i_4, j_4\}\})-\bar{E}(\{\{i_1, j_3\}, \{i_2, j_2\}, \{i_3, j_1\}, \{i_4, j_4\}\})\\
    & & -\bar{E}(\{\{i_1, j_2\}, \{i_2, j_1\}, \{i_3, j_3\}, \{i_4, j_4\}\})-\bar{E}(\{\{i_1, j_3\}, \{i_2, j_1\}, \{i_3, j_2\}, \{i_4, j_4\}\})\\
    & & -\bar{E}(\{\{i_1, j_1\}, \{i_2, j_2\}, \{i_3, j_3\}, \{i_4, j_4\}\})-\bar{E}(\{\{i_1, j_2\}, \{i_2, j_3\}, \{i_3, j_2\}, \{i_4, j_4\}\}).
\end{eqnarray*}
This gives us the relation 
$$\mathlarger{\mathlarger{\sum}}_{\tiny \sigma \in \mbox{Sym}\{1,2,3\}} \bar{E}(\{\{i_1, j_{\sigma(1)}\}, \{i_2, j_{\sigma(2)}\}, \{i_3, j_{\sigma(3)}\}, \{i_4, j_4\}\})=0.$$
By symmetry the case of having $I_3=\{i_2, i_3\}, I_4=\{i_4, j_4\}$ works in a similar manner. \\ \\
{\bf Case 4}: $I_3=\{j_1,j_2\}, I_4=\{j_3, j_4\}$.  We have
\begin{eqnarray*}
    0 & =&\bar{E}(\{\{i_1, i_2\}, \{i_3, i_4\}, \{j_1, j_2\}, \{j_3, j_4\}\})+\bar{E}(\{\{i_1, i_3\}, \{i_2, i_4\}, \{j_1, j_2\}, \{j_3, j_4\}\})\\
    & & +\bar{E}(\{\{i_1, i_4\}, \{i_2, i_3\}, \{j_1, j_2\}, \{j_3, j_4\}\}).
\end{eqnarray*}
Modulo Case 2, the choice of pairings does not matter, therefore we get
\begin{eqnarray*}
    0 &  =&-\bar{E}(\{\{i_1, j_1\}, \{i_2, j_2\}, \{i_3, i_4\}, \{j_3, j_4\}\})-\bar{E}(\{\{i_1, j_2\}, \{i_3, j_1\}, \{i_3, i_4\}, \{j_3, j_4\}\})\\
    & & -\bar{E}(\{\{i_1, j_1\}, \{i_3, j_2\}, \{i_2, i_4\}, \{j_3, j_4\}\})-\bar{E}(\{\{i_1, j_2\}, \{i_3, j_1\}, \{i_2, i_4\}, \{j_3, j_4\}\})\\
    & & - \bar{E}(\{\{i_1, j_1\}, \{i_4, j_2\}, \{i_2, i_3\}, \{j_3, j_4\}\})-\bar{E}(\{\{i_1, j_2\}, \{i_4, j_1\}, \{i_2, i_3\}, \{j_3, j_4\}\}).
\end{eqnarray*}
Cancelling the signs gives
\begin{align*}
    &\mathlarger{\mathlarger{\sum}}_{\tiny \sigma \in \mbox{Sym}\{2,3,4\}} \bar{E}(\{\{i_1, j_1\}, \{i_2, j_{\sigma(2)}\}, \{i_3, j_{\sigma(3)}\}, \{i_4, j_{\sigma(4)}\}\}) \\
    & + \mathlarger{\mathlarger{\sum}}_{\tiny \sigma \in \mbox{Sym}\{1,3,4\}} \bar{E}(\{\{i_1, j_2\}, \{i_2, j_{\sigma(1)}\}, \{i_3, j_{\sigma(3)}\}, \{i_4, j_{\sigma(4)}\}\}) = 0,
\end{align*}
which follows from Case 3.
We have seen that $W / W\cap J \cong W_0/W_0\cap J$, as vector spaces,  and $W_0\cap J$ is generated by relations 
\begin{equation}
    \mathlarger{\mathlarger{\sum}}_{\tiny \sigma \in \mbox{Sym}\{1,2,3\}} E(\{\{i_1, j_{\sigma(1)}\}, \{i_2, j_{\sigma(2)}\}, \{i_3, j_{\sigma(3)}\}, \{i_4, j_4\}, \dotsc, \{i_m, j_m\}\})=0;
\end{equation}
and 
{\small \begin{align}
    &\mathlarger{\sum}_{\tiny \sigma \in \mbox{Sym}\{1,2\}, \tau \in \mbox{Sym}\{3,4\}} E(\{\{i_1, j_{\sigma(1)}\}, \{i_2, j_{\sigma(2)}\}, \{i_3, j_{\tau(3)}\}, \{i_4, j_{\tau(4)}\}, \{i_5, j_5\}, \dotsc, \{i_m, j_m\}\})\\
   \nonumber  =& \mathlarger{\sum}_{\tiny \sigma \in \mbox{Sym}\{1,2\},  \tau \in \mbox{Sym}\{3,4\}} \bar{E}(\{\{i_1, j_{\tau(3)}\}, \{i_2, j_{\tau(4)}\}, \{i_3, j_{\sigma(1)}\}, \{i_4, j_{\sigma(2)}\}, \{i_5, j_5\}, \dotsc, \{i_m, j_m\}\});
\end{align}}

where $\{i_{1},\ldots ,i_{m}\}=\{2,3,\ldots ,m+1\}$ and $\{j_{1},\ldots ,j_{m}\}=\{m+2,\ldots ,2m+1\}$. We show that $W_{0}\cap J$ is a proper subspace of $W_{0}$ by showing that all the relators in (3) and (4) lie in a subspace of co-dimension $1$ in $W_{0}$. \\ \\
Let $M_0$ be the subspace of $W_0$ generated by all
                   $$E(\{\{2,\sigma(j_{1})\},\ldots ,\{m+1,\sigma(j_{m})\}\})+E(\{\{2,j_{1}\},\ldots ,\{m+1,j_{m}\}\})$$
with $\{j_{1},\ldots ,j_{m}\}=\{m+2,\ldots ,2m+1\}$ and where $\sigma$ is a transposition  in $\mbox{Sym\,}(\{m+2,\ldots ,2m+1\})$.  Thus modulo $M_{0}$ we have $E(\{\{2,\sigma(m+2)\},\ldots , \{m+1,\sigma(2m+1)\}\})=\mbox{sign}(\sigma)E(\{\{2,m+2\},\ldots ,\{m+1,2m+1\}\})$. 
Notice that $M_0$ is of codimension 1 in $W_0$ and contains $W_0\cap J$. For (3) this is because  $S_{3}$ has equally many even and odd permutations and for (4) we can transform the LHS to the RHS using even permutations. Thus $W_0/W_{0}\cap J \neq 0$ that implies that $W/W\cap J\neq 0$. \\ \\
As a consequence we get the following main  result of this section.

\begin{thm}
The ideal generated by $z$ in $L$ is non-nilpotent.
\end{thm}

\section{The normal subgroup $\langle x \rangle^G$}
\mbox{}\\
In Section 2 we considered the largest locally finite $p$-group $G = \langle x, a_1, a_{2}, \dots \rangle$ satisfying the following relations: \\
\begin{enumerate}
    \item $\langle a_i \rangle^G$ is abelian for all $i\geq 1$;
    \item  $\langle x \rangle^G$ is metabelian;
    \item $x^{p}=a_{1}^{p}=a_{2}^{p}=\cdots =1$;
    \item if $s \neq x$ is a simple commutator in $x,a_{1},a_{2},\ldots $, where $t(s) \geq 2$ or $t(s) \leq -2$, then $s=1$.
\end{enumerate}
\mbox{}\\
Now we use the Lie algebra $L$ constructed in the Section 3 to get a group $H=\langle 1+\ad(z),1+\ad(c_1),1+\ad(c_2), \dots \rangle$. In this section we will show that $H$ is a homomorphic image of $G$ by proving that the relations (1)-(4) hold in $H$ where
$x, a_{1}, a_{2},\ldots $ are replaced by $1+\ad(z), 1+\ad(c_1), 1+\ad(c_2),\ldots $ . We will also use the work in Section 2 and Section 3 to show that $\langle 1+\ad(z)\rangle^{H}$ is non-nilpotent from which follows that $\langle x\rangle^{G}$ is non-nilpotent.

\begin{lem}\label{L1}
We have $\ad(z)^2=0$ and $\ad(z)\ad(w)\ad(z)=0$ for any simple Lie product $w$ in $z, c_1, c_2, \dots$.
\end{lem}
\begin{proof}
Let $u$ be a simple Lie product in $z, c_1, c_2, \dots$. Notice first that if $u$ is non-trivial then the type of $[u,z,z]$ is less than or equal to $-3$ and therefore $[u,z,z]=0$. Turning to the second claim we know that  
$[u,z]=0$ if $u$ is of type different from $1$. We can therefore assume that $u$ is of type 1. Then if $w$ is non-zero, we have that $[u,z,w,z]$ is of type at most $-2$ and $[u,z,w,z]=0$.
\end{proof}

\begin{lem}
Let $w_H$ be a simple group commutator in $1+\ad(z), 1+\ad(c_1), 1+\ad(c_2), \dots$ and let $w_L$ be the corresponding Lie commutator in $z, c_1, c_2, \dots$. Then
$$
w_H=1+\ad(w_L).
$$
\end{lem}
\begin{proof}
We show this by using induction on the weight of $w_H$. If the weight is 1 then this is obvious. Now suppose the weight is $k \geq 2$ and the result holds whenever the weight is smaller. We have $w_H=[u_H, v_H]$, where $u_H, v_H$ are simple commutators of smaller weight in $1+\ad(z), 1+\ad(c_1), 1+\ad(c_2), \dotsc$. Let $u_L, v_L$ be the corresponding simple commutators in $z, c_1, c_2, \dotsc$. Using the induction hypothesis we have 
\begin{align*}
    w_H=& [u_H, v_H]=(1+\ad(u_L))^{-1} (1+\ad(v_L))^{-1} (1+\ad(u_L))(1+\ad(v_L))\\
    =& (1-\ad(u_L))(1-\ad(v_L))(1+\ad(u_L))(1+\ad(v_L)) \mbox{ (using Lemma \ref{L1})}\\
    =&1+\ad(u_L)\ad(v_L)-\ad(v_L)\ad(u_L)\\
    =&1+\ad([u_L, v_L])\\
    =&1+\ad(w_L).
\end{align*}
\end{proof}
Now we show that $H$ satisfies the relations (1)-(4) for $G$ where $x, a_1, a_2, \dots$ are replaced by  $1+\ad(z), 1+\ad(c_1), 1+\ad(c_2), \dots$. More precisely, we have the following. 

\begin{pr}
Let $H$ be as before. Then
\begin{enumerate}
    \item $\langle 1+\ad(c_i) \rangle^H$ is abelian for all $i\geq 1$;
    \item  $\langle 1+ \ad(z) \rangle^H$ is metabelian;
    \item $(1+\ad(z))^{p}=(1+\ad(c_{1}))^{p}=(1+\ad(c_{2}))^{p}=\cdots =1$;
    \item if $s \neq 1+\ad(z)$ is a simple commutator in $1+\ad(z),1+\ad(c_{1}),1+\ad(c_{2}),\ldots $, where $t(s) \geq 2$ or $t(s) \leq -2$, then $s=1$.
\end{enumerate}
\end{pr}

\begin{proof}
(1) Let $w_H$ be a simple group commutator in $1+\ad(z), 1+\ad(c_1), 1+\ad(c_2), \dots$ with a repeated occurrence of $1+\ad(c_i)$. Then $w_H=1+\ad(w_L)$ and as $\Id(c_i)$ is abelian, $\ad(w_L)=0$ and thus $w_H=1$. This implies that $\langle 1+\ad(c_i) \rangle^H$ is abelian for all $i\geq 1$. \\
(2) Let $u_H, v_H$ be simple commutators in $1+\ad(z), 1+\ad(c_1), 1+\ad(c_2), \dots$ where both have at least two occurrences of $1+\ad(z)$. Then
$$[u_H, v_H]=1+\ad([u_L, v_L])$$
and since $\Id(z)$ is metabelian, then $[u_L, v_L]=0$ and $[u_H, v_H]=1$. It follows then that $\langle 1+\ad(z) \rangle ^H$ is metabelian.\\
(3) Clearly $(1+\ad(z))^p=1+p\ad(z)=1$ and $(1+\ad(c_i))^p=1+p\ad(c_i)=1$.\\
(4) Let $w_H \neq 1+\ad(z)$ be a simple commutator in $1+\ad(z), 1+\ad(c_1), 1+\ad(c_2), \dotsc$, where $t(w_H) \geq 2$ or $t(w_H) \leqslant -2$. Then, $w_H=1+\ad(w_L)$, where $w_L$ is a Lie commutator in $z, c_1, c_2, \dotsc$, where $t(w_L) \geq 2$ or $t(w_L) \leqslant -2$. Here $w_L=0$ and so $w_H=1$. This completes the proof.
\end{proof}

\begin{thm}
The normal closure of $1+\ad(z)$ is not nilpotent.
\end{thm}

\begin{proof}
Consider 
\begin{align*}
    &[[1+\ad(z), 1+\ad(c_1), 1+\ad(c_2), 1+\ad(c_3)], [1+\ad(z), 1+\ad(c_4), 1+\ad(c_5)], \\ 
    & \dotsc, [1+\ad(z), 1+\ad(c_{2m}), 1+\ad(c_{2m+1})]]\\
    & =1+\ad([[z, c_1, c_2, c_3], [z, c_4, c_5], \dotsc, [z, c_{2m}, c_{2m+1}]]).
\end{align*}
Notice that 
$$1+\ad([[z, c_1, c_2, c_3], [z, c_4, c_5], \dotsc,[z, c_{2m}, c_{2m+1}]]) \neq 1,$$
since, for example 
$[[z, c_1, c_2, c_3], [z, c_4, c_5], \dotsc, [z, c_{2m}, c_{2m+1}], z]) \neq 0.$
Thus, we have shown that there is a simple commutator of arbitrary weight $m$ in $(1+\ad(z))$ that is non-trivial. This finishes the proof.
\end{proof}
{\it Acknowledgement}. We acknowledge the EPSRC (grant number 16523160) for support. We also thank Gareth Tracey and James Williams for participating in discussions at an early stage of this project. In particular we would like to thank the latter for some useful GAP calculations.

\end{document}